\documentclass{Rsepubli}

\usepackage[colorlinks=true,urlcolor=blue,
citecolor=red,linkcolor=blue,linktocpage,pdfpagelabels,
bookmarksnumbered,bookmarksopen]{hyperref}

\newcommand{\h}{H^{1}(\mathbb{R}_+^{N+1})}

\newcommand{\permanent}[1]{\footnote{Permanent address: #1}}

\newcommand{\email}[1]{{\upshape(\texttt{#1})}}

\newtheorem{remark}[theorem]{Remark}

\begin{document}

\title[Ground states for the pseudo-relativistic Hartree equation]{Ground states for the pseudo-relativistic Hartree equation with external potential}
\author[S. Cingolani and S. Secchi]{Silvia Cingolani\permanent{Dipartimento di Meccanica, Matematica e Management,
Politecnico di Bari. Via Orabona 4, 70125 Bari (Italy)}\email{s.cingolani@poliba.it} \\ Simone Secchi\permanent{Dipartimento di matematica e applicazioni, Universit\`a di Milano Bicocca. Via R. Cozzi 53, 20125 Milano (Italy)}\email{simone.secchi@unimib.it}}
\maketitle

\begin{abstract}

We prove existence of positive ground state solutions to the pseudo-relativistic Schr\"{o}dinger equation
\begin{equation*}
\left\{
\begin{array}{l}
\sqrt{-\Delta +m^2} u +Vu = \left( W *
|u|^{\theta} \right)|u|^{\theta -2} u \quad\text{in $\mathbb{R}^N$}\\
u \in H^{1/2}(\mathbb{R}^N)
\end{array}
\right.
\end{equation*}
where $N \geq 3$,  $m >0$, $V$ is a bounded external scalar potential and $W$ is a convolution potential, radially symmetric, satisfying suitable assumptions. We also furnish some asymptotic decay estimates of the found solutions.
\end{abstract}

\section{Introduction}
The mean field limit of a quantum system describing many self-gravitating, relativistic bosons  with rest mass $m>0$
leads to the time-dependent pseudo-relativistic Hartree equation
\begin{equation}\label{eq:1.1}
i\frac{\partial \psi}{\partial t} = \left(\sqrt{-\Delta +m^2} -m \right)\psi -  \left( \frac{1}{|x|} * |\psi|^2 \right)  \psi, \quad x \in \mathbb{R}^3
\end{equation}
where  $\psi : \mathbb{R} \times \mathbb{R}^3 \to \mathbb{C}$ is the wave field. Such a physical system is often referred to as a boson star in astrophysics (see \cite{es,fl,fjl}).
Solitary wave solutions $\psi(t,x)= e^{-i t \lambda} \phi$, $\lambda \in \mathbb{R}$ to equation (\ref{eq:1.1}) satisfy the equation
\begin{equation}\label{psedorelativistichartree}
\left( \sqrt{-\Delta +m^2} -m \right)\phi - \left( \frac{1}{|x|} * |\phi|^2 \right)  \phi
 =\lambda \phi
\end{equation}

For the non-relativistic Hartree equation, existence and uniqueness (modulo translations) of a minimizer has been prove by Lieb  \cite{lieb} by using symmetric decreasing rearrangement inequalities.
Within the same setting, always for the negative Laplacian, Lions \cite{l2} proved existence of infinitely
many spherically symmetric solutions by application of abstract critical point theory both
without the constraint  and with the constraint for a more general radially symmetric convolution potential.
The  non-relativistic Hartree equations is also known as the Choquard-Pekard or Schr\"odinger-Newton equation and recently a large amount of papers are devoted to the study of solitary states and its semiclassical limit
(see \cite{a,cho,ccs,ccs1,css,fl,ls,mpt,mz,pe2,pe3,s,t,ww} and references therein).

In \cite{ly} Lieb and Yau  solved the pseudo-relativistic Hartree  equation $(\ref{psedorelativistichartree})$ by minimization on the sphere $\{\phi \in L^2(\mathbb{R}^3) \mid \int_{\mathbb{R}^3} |\phi|^2 =M \}$, and they proved that a radially symmetric ground state exists in $H^{1/2}(\mathbb{R}^3)$ whenever $M<M_c$, the so-called Chandrasekhar mass. These results have been generalized in \cite{CNbis}. Later Lenzmann proved in \cite{Lenz} that this ground state is unique (up to translations and phase change) provided that the mass $M$ is sufficiently small; some results about the non-degeneracy of the ground state solution are also given.

Quite recently, Coti Zelati and Nolasco (see \cite{CN}) studied the equation
\[
\sqrt{-\Delta +m^2}u=\mu u +\nu |u|^{p-2}u+\sigma (W*u^2)u
\]
under the assumptions that $p \in \left(2,\frac{2N}{N-1} \right)$, $N \geq 3$, $\mu <m$, $m >0$, $\nu \geq 0$, $\sigma \geq 0$ but not both zero, $W \in L^r(\mathbb{R}^N)+L^\infty(\mathbb{R}^N)$, $W\geq 0$, $r > N/2$, $W$ is radially symmetric and decays to zero at infinity. They proved the existence of a positive, radial solution that decays to zero at infinity exponentially fast. For the case $\sigma < 0$, $\mu < m $,  we also refer to \cite{Mugnai} where a more general nonlinear term is considered.

In the present work we consider a generalized pseudo-relativistic Hartree equation
\begin{equation}\label{eq:1}
\sqrt{-\Delta +m^2} u +V u = \left( W *
|u|^{\theta} \right)|u|^{\theta -2} u \quad\text{in $\mathbb{R}^N$}
\end{equation}
where $N \geq 3$, $m >0$, $V$ is an external potential and $W \geq 0$  is a radially symmetric convolution kernel such that
$\lim_{|x| \to +\infty} W(|x|) = 0$.

In \cite{MeZo} Melgaard and Zongo prove that $(\ref{eq:1})$ has  a sequence of radially symmetric solutions of higher and higher energy, assuming that $V$ is radially symmetric potential, $\theta=2$ and under some restrictive assumptions on the structure of the kernel $W$.

Here we are interested to look for positive ground state solutions for the pseudo-relativistic Hartree equation  $(\ref{eq:1})$
when $V$ is not symmetric. In such a case, the nonlocality of $\sqrt{-\Delta +m^2}$ and the presence of the external potential $V$ (not symmetric) complicate the analysis of the pseudo-relativistic Hartree equation in a substantial way. The main difficulty is, as usual, the lack of compactness.

\medskip
In what follows, we will keep the following assumptions:

\begin{itemize}
\item[(V1)] $V\colon \mathbb{R}^N \to \mathbb{R}$ is a continuous and bounded function, and $V(y)+V_0 \geq 0$ for every $y \in \mathbb{R}^N$ and for some $V_0\in (0,m)$.
\item[(V2)] There exist $R>0$ and $k \in (0,2m)$ such that
\begin{equation}\label{eq:2} V(x) \leq V_\infty - e^{-k|x|}
\quad\text{for all $|x| \geq R$}
\end{equation}
where $V_\infty = \liminf_{|x| \to +\infty} V(x)>0$.
\item[(W)]
$W \in L^r(\mathbb{R}^N)+L^\infty
(\mathbb{R}^N)$ for some $r> \max \left\{ 1, \frac{N}{N(2-\theta) + \theta} \right\}$ and $2 \leq \theta <\frac{2N}{N-1}$.
\end{itemize}

Our main result is the following.
\begin{theorem} \label{th:main}
Retain assumptions (V1), (V2) and
(W). Then equation (\ref{eq:1}) has at
least a positive solution $u \in H^{1/2}(\mathbb{R}^N)$.
\end{theorem}

We remark that Theorem \ref{th:main} applies for a large class of bounded electric potentials without symmetric constraints and  covers the physically relevant cases of Newton or Yukawa type two body interaction, i.e. $W(x)= \frac{1}{|x|^\lambda}$ with $0 < \lambda < 2$, $W(x)= \frac{e^{-|x|}}{|x|}$.

The theorem will be proved using variational methods. Firstly we transform the problem into an elliptic equation with  nonlinear Neumann boundary conditions, using a {\sl local} realization of the pseudo-differential operator $\sqrt{- \Delta + m^2}$ as in \cite{CSM, CT, CN, CNbis}.
 The corresponding solutions are found as critical points of an Euler functional defined in $H^1(\mathbb{R}_+^{N+1})$. We show that such a functional satisfies the Palais-Smale condition below some energy level determined by the value of $V$ at infinity and we prove the existence of a Mountain Pass solution under the level where the Palais-Smale condition holds.  Finally  we also furnish some asymptotic decay estimates of the found solution.

\smallskip
Local and global well-posedness results for pseudo-relativistic Hartree equations with external potential are proved by Lenzmann in \cite{Lenz2}.

\section{The variational framework}

Before we state our main result, we recall a few basic facts about the functional setting of our problem.
The operator $\sqrt{-\Delta +m^2}$ can be defined by Fourier analysis: given any $\phi \in L^2(\mathbb{R}^3)$ such that
\[
\int_{\mathbb{R}^3} \left(m^2 + |k|^2 \right) |\mathcal{F}\phi|^2 \, dk < +\infty,
\]
we define $\sqrt{-\Delta + m^2} \phi$ via the identity
\[
\mathcal{F} \left(\sqrt{-\Delta + m^2} \phi\right) = \sqrt{m^2+|k|^2} \mathcal{F}\phi,
\]
$\mathcal{F}$ being the usual Fourier transform. The condition
\[
\int_{\mathbb{R}^3}\sqrt{m^2+|k|^2} |\mathcal{F}\phi|^2 \, dk < +\infty
\]
is known to be equivalent to $\phi \in H^{1/2}(\mathbb{R}^3)$. In this sense, the fractional Sobolev
space $H^{1/2}(\mathbb{R}^3)$ is the natural space to work in. However, this definition is not particularly
convenient for variational methods, and we prefer a \emph{local} realization of the operator in the augmented half-space, originally inspired by the paper~\cite{CS} for the fractional laplacian.

Given $u \in \mathcal{S}(\mathbb{R}^N)$,
the Schwarz space of rapidly decaying smooth functions defined on
$\mathbb{R}^N$, there exists one and only one function $v \in
\mathcal{S}(\mathbb{R}^{N+1}_{+})$ (where $\mathbb{R}^{N+1}_{+}
=(0,+\infty) \times \mathbb{R}^N$) such that
\begin{equation*}
\begin{cases} -\Delta v + m^2 v =0 &\text{in $\mathbb{R}^{N+1}_{+}$}\\
v(0,y)=u(y) &\text{for $y \in \mathbb{R}^N = \partial
\mathbb{R}^{N+1}_{+}$}.
\end{cases}
\end{equation*} Setting
\[
Tu(y)=-\frac{\partial v}{\partial x}(0,y),
\]
we easily see that the problem
\begin{equation*}
\begin{cases} -\Delta w + m^2 w =0 &\text{in $\mathbb{R}^{N+1}_{+}$}\\
w(0,y)=Tu(y) &\text{for $y \in \partial \mathbb{R}^{N+1}_{+} = \{ 0\}
\times \mathbb{R}^N \simeq \mathbb{R}^N$}
\end{cases}
\end{equation*}
is solved by $w(x,y)=-\frac{\partial v}{\partial
x}(x,y)$. From this we deduce that
\[
T(Tu)(y)=-\frac{\partial w}{\partial x}(0,y) = \frac{\partial^2
v}{\partial x^2}(0,y) = \left(-\Delta_y v + m^2 v \right)(0,y),
\]
and hence $T\circ T = (-\Delta_y +m^2)$, namely $T$ is a square
root of the Schr\"{o}dinger operator $-\Delta + m^2 $ on
$\mathbb{R}^{N}=\partial \mathbb{R}_{+}^{N+1}$.

\vspace{10pt}

In the sequel, we will write
$|\cdot|_p$ for the norm in $L^p(\mathbb{R}^N)$ and $\|\cdot\|_p$ for
the norm in $L^p(\mathbb{R}^{N+1}_{+})$. The symbol $\| \cdot \|$ will
be reserved for the usual norm of $H^1(\mathbb{R}_{+}^{N+1})$.

The theory of traces for Sobolev spaces ensures
that every function $v \in H^1(\mathbb{R}^{N+1}_{+})$
possesses a \emph{trace} $\gamma(v) \in H^{1/2}(\mathbb{R}^N)$ which
satisfies the inequality (see \cite[Lemma 13.1]{Tartar})
\begin{equation}\label{eq:5} |\gamma(v)|_p^p \leq p
\|v\|_{2(p-1)}^{p-1} \left\| \frac{\partial v}{\partial x}
\right\|_{2}
\end{equation}
whenever $2 \leq p \leq 2N/(N-1)$.
This also implies that, for every $\lambda >0$,
\begin{equation} \label{eq:19}
\int_{\mathbb{R}^N} \gamma(v)^2 \leq \lambda \int_{\mathbb{R}_{+}^{N+1}}|v|^2 + \frac{1}{\lambda} \int_{\mathbb{R}_{+}^{N+1}} \left| \frac{\partial v}{\partial x} \right|^2.
\end{equation}
As a particular case, we record
\begin{equation} \label{eq:21}
\int_{\mathbb{R}^N} \gamma(v)^2 \leq m \int_{\mathbb{R}_{+}^{N+1}}|v|^2 + \frac{1}{m} \int_{\mathbb{R}_{+}^{N+1}} \left| \nabla v \right|^2.
\end{equation}
It is also known (see \cite[Lemma 16.1]{Tartar})
that any element of $H^{1/2}(\mathbb{R}^N)$ is the trace of some
function in $H^1(\mathbb{R}_{+}^{N+1})$.

\vspace{10pt}
From the previous construction and following \cite{CSM, CT, CN, CNbis}, we can replace the \emph{nonlocal} problem
(\ref{eq:1}) with the local Neumann problem
\begin{equation}\label{eq:3}
\begin{cases} -\Delta v + m^2 v =0 &\text{in $\mathbb{R}^{N+1}_{+}$}\\
-\frac{\partial v}{\partial x} = -V(y)v +
(W*|v|^{\theta})|v|^{\theta-2}v &\text{in $\mathbb{R}^N = \partial
\mathbb{R}^{N+1}_{+}$}.
\end{cases}
\end{equation}
%
%
We will look for solutions to (\ref{eq:3}) as critical points of the
Euler functional \newline $I \colon H^1(\mathbb{R}_{+}^{N+1}) \to \mathbb{R}$ defined by
\begin{multline}\label{eq:4}
I(v) = \frac{1}{2}
\iint_{\mathbb{R}^{N+1}_{+}} \left(|\nabla v|^2 + m^2 v^2 \right) dx
\, dy + \frac{1}{2}\int_{\mathbb{R}^N} V(y)\gamma(v)^2\, dy \\
{}-
\frac{1}{2 \theta} \int_{\mathbb{R}^N} (W * |\gamma(v)|^{\theta})
|\gamma(v)|^{\theta}\, dy.
\end{multline}
We recall the well-known Young's inequality.

\begin{proposition}
Assume   $f \in L^p(\mathbb{R}^N)$, $g \in L^q(\mathbb{R}^N)$ with $1 \leq p$, $q, r \leq \infty$, $p^{-1}+q^{-1}=1 + r^{-1}$.
Then
\begin{equation} \label{eq:hlss}
| f * g|_r  \leq  |f|_p |g|_q.
\end{equation}
\end{proposition}

A major tool for our analysis is the generalized Hardy-Littlewood-Sobolev inequality. We recall that $L_w^q(\mathbb{R}^N)$ is the \emph{weak} $L^q$ space: see \cite{Lieb1983} for a definition. We denote by $| \cdot|_{q,w}$ the usual norm in $L_w^q(\mathbb{R}^N)$.
\begin{proposition}[\cite{Lieb1983}]
Assume that $p$, $q$ and $t$ lie in $(1,+\infty)$ and $p^{-1}+q^{-1}+t^{-1}=2$. Then, for some constant $N_{p,q,t}>0$ and for any $f \in L^p(\mathbb{R}^N)$, $g \in L^t(\mathbb{R}^N)$ and $h \in L_w^q(\mathbb{R}^N)$, we have the inequality
\begin{equation} \label{eq:hls}
\left| \int f(x)h(x-y)g(y)\, dx\, dy
\right| \leq N_{p,q,t} |f|_p |g|_t |h|_{q,w}.
\end{equation}
\end{proposition}
For the sake of completeness, we check that the functional $I$ is well-defined.
As a consequence of (\ref{eq:5}), for every $p \in \left[2,
\frac{2N}{N-1}\right]$ we deduce that
\begin{equation}\label{eq:6} |\gamma(v)|_p \leq \frac{p-1}{p}
\|v\|_{2(p-1)} + \|\nabla v \|_2 \leq C_p \|v\|,
\end{equation} and the term $\int_{\mathbb{R}^N} V(y) \gamma(v)^2\,
dy$ in the expression of $I(v)$ is finite because of the boundedness
of $V$.
Writing $W=W_1+W_2 \in L^r(\mathbb{R}^N) + L^\infty(\mathbb{R}^N)$  and using $(\ref{eq:hls})$ we
can estimate the convolution term as follows:
\begin{multline} \label{eq:15}
\int_{\mathbb{R}^N} (W *
|\gamma(v)|^\theta)|\gamma(v)|^\theta = \int_{\mathbb{R}^N} (W_1 *
|\gamma(v)|^\theta)|\gamma(v)|^\theta + \int_{\mathbb{R}^N} (W_2 *
|\gamma(v)|^\theta)|\gamma(v)|^\theta \\ \leq |W_1|_{r}
|\gamma(v)|_{\frac{2r\theta}{2r-1}}^{2\theta} + |W_2|_\infty
|\gamma(v)|_{\theta}^{2\theta} \leq |W_1|_{r} \|v\|^{2\theta} +
|W_2|_\infty \|v\|^{2\theta}.
\end{multline}
Since
\[
r > \frac{N}{N(2-\theta) + \theta} \quad\text{and}\quad 2 \leq \theta < \frac{2N}{N-1},
\]
there results
\[
\frac{2 r -1 }{2 \theta r} = \frac{1}{\theta} -  \frac{1}{2\theta r} > \frac{1}{\theta}-\frac{N(2-\theta)+\theta}{2\theta N} = \frac{N-1}{2N},
\]
and thus
\[
 \frac{2\theta r}{2r-1} < \frac{2N}{N-1},
\]
and from (\ref{eq:15}) we see that the convolution term in $I$ is finite. It is easy to check, by the same token, that
$I \in C^1(H^1(\mathbb{R}_{+}^{N+1}))$.
\begin{remark}
 As we have just seen, estimates involving the kernel $W$ always split into two parts. As a rule, those with the \emph{bounded}
 kernel $W_2 \in L^\infty (\mathbb{R}^N)$ are straightforward. In the sequel, we will often focus on the contribution of
 $W_1 \in L^r (\mathbb{R}^N)$ and drop the easy computation with $W_2$.
\end{remark}

\section{The limit problem}
Let us consider the space of the symmetric functions
\[ H^{\sharp} = \{ u \in \h \mid \text{$u(x, Ry)= u(x,y)$ for all $R \in
O(N)$} \}.
\]
Let us consider the functional $J_\alpha \colon H^\sharp \to
\mathbb{R}$ defined by setting
\begin{multline}\label{eq:12}
J_\alpha (v) = \frac{1}{2}
\iint_{\mathbb{R}^{N+1}_{+}} \left(|\nabla v|^2 + m^2 v^2 \right) dx
\, dy + \frac{1}{2}\int_{\mathbb{R}^N} \alpha \gamma(v)^2\, dy \\
{}-
\frac{1}{2 \theta} \int_{\mathbb{R}^N} (W * |\gamma(v)|^{\theta})
|\gamma(v)|^{\theta}\, dy.
\end{multline}
where $W \geq  0$ is radially symmetric, $\lim_{|x| \to +\infty} W(|x|)=0$ and assumption $(W)$ holds.
If $\alpha > -m$,  we can extend the arguments in Theorem 4.3 \cite{CN}, for the case $\theta=2$, and prove that the functional
$J_\alpha$ has a Mountain Pass critical point $v_\alpha \in H^\sharp$,
namely
\begin{equation}\label{Mountainpass} J_\alpha (v_\alpha)= E_\alpha =
\inf_{g \in \Gamma_\sharp} \max_{t \in [0,1]} J_\alpha (g(t))
\end{equation} where $\Gamma_\sharp =\{ g \in C([0,1]; H^\sharp) \mid g(0)
=0,\ J_\alpha(g(1)) <0 \}$.
The critical point $v_\alpha$ corresponds to a weak solution of
\begin{equation}\label{5}
  \begin{cases} -\Delta v + m^2 v =0 &\text{in
$\mathbb{R}^{N+1}_{+}$}\\ -\frac{\partial v}{\partial x} = - \alpha v
+ ( W*|v|^{\theta})|v|^{\theta-2}v &\text{in $\mathbb{R}^N = \partial
\mathbb{R}^{N+1}_{+}$}.
  \end{cases}
\end{equation}

In the sequel, we need a standard characterization of the
mountain-pass level $E_\alpha$. Let us define the \emph{Nehari
manifold} $\mathcal{N}_\alpha$ associated to the functional $J_\alpha$:
\begin{multline} \label{eq:9}
\mathcal{N}_\alpha = \Big\{ v \in
H^{\sharp} \mid \iint_{\mathbb{R}^N \times \mathbb{R}^N} |\nabla v|^2 +
m^2 v^2 \, dx \, dy \\
= -\alpha \int_{\mathbb{R}^N} \gamma(v)^2 \, dy +
\int_{\mathbb{R}^N} \left(W*|\gamma(v)|^\theta \right)|\gamma(v)|^{\theta} \, dy
\Big\}.
\end{multline}
\begin{lemma} There results
\begin{equation}\label{eq:10} \inf_{v \in \mathcal{N}_\alpha} J_\alpha
(v) = \inf_{v \in H^{\sharp}} \max_{t >0} J_\alpha (tv) = E_\alpha.
\end{equation}
\end{lemma}
\begin{proof}
The proof is straightforward, since $J_\alpha$ is the sum of
homogeneous terms; we follow \cite{W}.
First of all, for $v \in H^\sharp$ we compute
\begin{multline} \label{eq:11}
J_\alpha (tv) = \frac{t^2}{2} \left(
\iint_{\mathbb{R}^N \times \mathbb{R}^N} |\nabla v|^2 + m^2 v^2 \, dx \,
dy + \alpha \int_{\mathbb{R}^N} \gamma(v)^2 \, dy \right) \\
{}- \frac{t^{2
\theta}}{2 \theta} \int_{\mathbb{R}^N} \left( W* |\gamma(v)|^\theta \right)
|\gamma(v)|^{\theta} \, dy.
\end{multline}
Since $\theta \geq 2$ and by using (\ref{eq:21}),
it is easy to check that $t \in
(0,+\infty) \mapsto J(tv)$ possesses a unique critical point
$t=t(v)>0$ such that $t(v)v \in \mathcal{N}_\alpha$. Moreover, since
$J_\alpha$ has the mountain-pass geometry, $t=t(v)$ is a maximum
point. It follows that
\[ \inf_{v \in \mathcal{N}_\alpha} J_\alpha (v) = \inf_{v \in H^{\sharp}}
\max_{t >0} J_\alpha (tv).
\] The manifold $\mathcal{N}_\alpha$ splits $H^\sharp$ into two connected
components, and the component containing $0$ is open. In addition, $J_\alpha$
is non-negative on this component, because $\langle J'_\alpha (tv),v
\rangle \geq 0$ when $0<t \leq t(v)$. It follows immediately that any
path $\gamma \colon [0,1] \to H^\sharp$ with $\gamma (0)=0$ and $J_\alpha
(\gamma(1))<0$ must cross $\mathcal{N}_\alpha$, so that
\[ E_\alpha \geq \inf_{v \in \mathcal{N}_\alpha}J_\alpha (v).
\] The proof of (\ref{eq:10}) is complete.
\end{proof}

Following completely analogous argument in Theorem 3.14 and Theorem 5.1 in \cite{CN}, we can state the following result.

\begin{theorem} \label{th:CZN}
Let $\alpha + m >0 $ and  $(W)$ holds. Then
$v_\alpha \in C^\infty([0, + \infty) \times \mathbb{R}^N)$, $v_\alpha(x,y)
>0$ in $[0, \infty) \times \mathbb{R}^N$ and for any $0 \leq \sigma
\in (-\alpha, m)$ there exists $C>0$ such that
\[ 0 < v_\alpha(x,y) \leq C e^{-(m - \sigma) \sqrt{x^2 + |y|^2}} \,
e^{-\sigma x}
\] for all $(x,y) \in [0, + \infty) \times \mathbb{R}^N$. In particular,
\[
0<v_\alpha (0,y)  \leq C e^{-\delta |y|} \quad\text{for every $y \in \mathbb{R}^N$}
\]
where $0<\delta <m+\alpha$ if $\alpha \leq 0$, and $\delta = m$ if $\alpha >0$.
\end{theorem}

\section{The Palais-Smale condition}

For any $v \in \h$ we denote
\[
\mathbb{D}(v)=\iint_{\mathbb{R}^{N}\times
\mathbb{R}^N}W(x-y)|\gamma(v)(x)|^{\theta}
|\gamma(v)(y)|^{\theta}dxdy.
\]
Inequality (\ref{eq:15}) yields immediately
\begin{equation} \label{Dbdd}
\mathbb{D}(v)\leq K\Vert v \Vert^{2 \theta}
\end{equation}
for every $v\in \h.$
\begin{lemma}
\label{lemPSto0}
Let $\{v_{n}\}_n$ be a sequence in $H^{1}(\mathbb{R}_+
^{N+1})$ such that $v_{n}\rightharpoonup0$ weakly in $\h$
\[ I(v_{n})\rightarrow c < E_{V_{\infty}}\text{\quad and \quad}
I'(v_n) \rightarrow0
\]
where $V_{\infty}:=\liminf\limits_{\left\vert x\right\vert
\rightarrow\infty}V(x)>0.$ Then a subsequence of $\{v_{n}\}_n$ converges
strongly to $0$ in $\h$.
\end{lemma}
\begin{proof}
First of all, we recall (\ref{eq:21}), and we rewrite $I(v)$ as
\begin{equation*}
I(v) = \frac{1}{2} \int_{\mathbb{R}_{+}^{N+1}} |\nabla v|^2 + m^2 |v|^2 + \frac{1}{2} \int_{\mathbb{R}^N} \left(V+V_0 \right) \gamma(v)^2 - \frac{V_0}{2} \int_{\mathbb{R}^N} \gamma(v)^2 - \frac{1}{2\theta}\mathbb{D}(v),
\end{equation*}
so that $V+V_0 \geq 0$ everywhere. Now,
\begin{equation}\label{pri}
c+1+\|v_n\| \geq I(v_n) - \frac{1}{2} \langle I'(v_n),v_n \rangle = \left( \frac{1}{2} -\frac{1}{2\theta} \right) \mathbb{D}(v_n),
\end{equation}
which implies that, for some constants $C_1$ and $C_2$,
\[
\frac{1}{2\theta}\mathbb{D}(v_n) \leq C_1 \|v_n\| +C_2.
\]
But then, using (\ref{eq:21}),
\begin{align*}
c+1 &\geq I(v_n) \\
&\geq \frac{1}{2} \int_{\mathbb{R}_{+}^{N+1}} |\nabla v_n|^2 + \frac{m^2}{2} \int_{\mathbb{R}_{+}^{N+1}} | v_n|^2 -\frac{V_0}{2} \left( m \int_{\mathbb{R}_{+}^{N+1}} |v_n|^2 + \frac{1}{m} \int_{\mathbb{R}_{+}^{N+1}} |\nabla v_n|^2  \right) \\
&\qquad {} -C_1 \|v_n\| -C_2\\
&= \frac{1}{2} \left( 1 - \frac{V_0}{m} \right) \int_{\mathbb{R}_{+}^{N+1}} |\nabla v_n|^2 +\frac{m(m-V_0)}{2} \int_{\mathbb{R}_{+}^{N+1}} |v_n|^2-C_1 \|v_n\| -C_2
\end{align*}
and since $m-V_0>0$ we deduce that $\{v_n\}$ is a bounded sequence in $H^1(\mathbb{R}_{+}^{N+1})$.

A standard argument shows that $\|v_n\|$ is bounded in
$\h$ and
\[
\frac{\theta-1}{2\theta} \left(\| v_{n}\|^{2} +
\int_{\mathbb{R}^{N}} V(y) \gamma(v_n)^2 dy \right) \rightarrow c\text{\quad and \quad } \frac{\theta-1}{2\theta}\mathbb{D}(v_{n})\rightarrow c.
\]
Therefore $c \geq 0$. If $c=0$, then
\begin{align*}
o(1) &=
 \left(\| v_{n}\|^{2} +
\int_{\mathbb{R}^{N}} V(y) \gamma(v_n)^2 dy \right) \\
&\geq
\left( 1 - \frac{V_0}{m} \right) \int_{\mathbb{R}_{+}^{N+1}} |\nabla v_n|^2 + {m(m-V_0)} \int_{\mathbb{R}_{+}^{N+1}} |v_n|^2,
\end{align*}
and $m -V_0 >0$ yields that $v_n \to 0$ strongly in $\h$.

Assume therefore that $c>0.$ Fix $\alpha <V_{\infty}$ such that $c< E_{\alpha},$
and $R_{0}>0$ such that $V(x)\geq\alpha$ if $\left\vert x\right\vert
\geq R_{0}.$ Let $\varepsilon\in(0,1).$ Since $\{v_{n}\}_n$ is bounded in
$\h$ there exists $R_{\varepsilon}>R_{0}$ such that
$R_{\varepsilon}\rightarrow + \infty$ as $\varepsilon\rightarrow0$
and, after passing to a subsequence,%
\begin{equation}\label{guscio} \iint_{S_{R_ \epsilon}} \left(|\nabla
v_n|^2 + m^2 v_n^2 \right) dx \, dy + \int_{ A_{R_ \epsilon}}
V(y)\gamma(v_n)^2\, dy <\varepsilon\text{ \ \ for all }n\in\mathbb{N}.
\end{equation} where
\begin{align*}
S_{R_ \epsilon} &= \{ z=(x,y) \in {\mathbb{R}_+^{N+1}}
\mid {R_{\varepsilon}<\left\vert z \right\vert <R_{\varepsilon}+1} \}
\\
A_{R_ \epsilon} &= \{ y \in {\mathbb{R}^{N}} \mid
{R_{\varepsilon}<\left\vert y \right\vert <R_{\varepsilon}+1} \}.
\end{align*}
If this is not the case, for any $m \in \mathbb{N}$, $m
\geq R_0$ there exists $\nu(m) \in\mathbb{N}$ such that
\begin{equation}\label{guscio1} \iint_{S_m} \left(|\nabla v_n|^2 + m^2
v_n^2 \right) dx \, dy + \int_{ A_m} V(y)\gamma(v_n)^2\, dy \geq
\varepsilon
\end{equation} for any $n\in\mathbb{N}$, $n \geq \nu(m)$. We can
assume that $\nu(m)$ is nondecreasing.  Therefore for any integer $m
\geq R_0$ there exists an integer $\nu(m)$ such that
\begin{multline}
\|v_n\|^2 + \int_{\mathbb{R}^{N}} V(y)\gamma(v_n)^2\,
dy \geq \iint_{T_m} \left(|\nabla v_n|^2 + m^2 v_n^2 \right) dx \, dy
+ \int_{ B_m} V(y)\gamma(v_n)^2\, dy \\
\geq (m - R_0)
\varepsilon \label{contr}
\end{multline}
for any $n \geq \nu(m)$, where $T_m = \{ z=(x,y) \in
{\mathbb{R}_+^{N+1}} \mid R_{0}<\left\vert z \right\vert < m \}$ and
$B_m = \{ y \in {\mathbb{R}^{N}} \mid R_0 <\left\vert y \right\vert <
m \},$ which contradicts the fact that $\| v_n\|$ is bounded.

We may assume that $\left\vert v_{n}\right\vert \rightarrow0$ strongly
in $L_{\mathrm{loc}}^{p}(\mathbb{R}^{N})$ with $p<2N/(N-1)$
and thus $\left\vert \gamma( v_{n}) \right\vert \rightarrow0$ strongly
in $L_{\mathrm{loc}}^{p}(\mathbb{R}^{N})$.

Let $\xi_\varepsilon \in C^{\infty}(\mathbb{R}_+^{N+1})$ be a symmetric function,
namely $\xi_\varepsilon(x, gy) = \xi_\varepsilon(x,y)$ for all $g \in O(N)$, $x >0$, $y \in
\mathbb{R}^{N}$.  Moreover assume that $\xi_\varepsilon(z)= 0$ if $|z| \leq
R_\varepsilon$ and $\xi_\varepsilon(z)= 1$ if $|z| \geq R_\varepsilon + 1$ and
$\xi(z)\in\lbrack0,1]$ for all $z \in\mathbb{R}_{+}^{N+1}$. Set $w_n= \xi_\varepsilon
v_n$.
We apply now Young's inequality (\ref{eq:hls}) with $p=q=2r/(2r-1)$ and $h=W$, $f=|\gamma(v_n)|^\theta$, $g=|\gamma(v_n)|^\theta - |\gamma(w_n)|^\theta$:
\begin{gather} \label{difD}
\left\vert
\mathbb{D}(v_{n})-\mathbb{D}(w_{n})\right\vert  \nonumber \\
\leq
\iint_{\mathbb{R}^{N}\times \mathbb{R}^N} W(x-y)\left\vert
|\gamma(v_{n}) (x)|^{\theta}|
\gamma(v_{n})(y)|^{\theta}-|\gamma(w_{n})(x)|^{\theta}|\gamma(w_{n})(y)|^{\theta}\right\vert dxdy \nonumber \\
= \iint_{\mathbb{R}^{N}\times \mathbb{R}^N} W(x-y) \Big| |\gamma(v_{n}) (x)|^{\theta}|
\gamma(v_{n})(y)|^{\theta}-|\gamma(v_n)(x)|^\theta |\gamma(w_n)(y)|^\theta \nonumber \\
\quad + |\gamma(v_n)(x)|^\theta |\gamma(w_n)(y)|^\theta-|\gamma(w_{n})(x)|^{\theta}|\gamma(w_{n})(y)|^{\theta}  \Big| dx\, dy \nonumber\\
\leq \iint_{\mathbb{R}^{N}\times \mathbb{R}^N} W(x-y) |\gamma(v_n)(x)|^\theta \left| |\gamma(v_n)(y)|^\theta - |\gamma(w_n)(y)|^\theta \right|\, dx\, dy \nonumber \\
\quad + \iint_{\mathbb{R}^{N}\times \mathbb{R}^N} W(x-y) |\gamma(w_n)(y)|^\theta \left| |\gamma(v_n)(x)|^\theta - |\gamma(w_n)(x)|^\theta \right| \, dx\, dy
\nonumber \\
 \leq 2
\iint_{\mathbb{R}^N \times \mathbb{R}^{N}}
W(x-y) |\gamma(v_{n})(x)|^{\theta} \left\vert
|\gamma(v_{n})(y)|^{\theta}-|\gamma(w_{n})(y)|^{\theta}\right\vert
dxdy\nonumber \\
 \leq 2C |W|_r \left|\gamma(v_n)\right|^{\theta}_{\frac{2r \theta}{2r-1}} \left| |\gamma(v_n)|^\theta - |\gamma(w_n)|^\theta
\right|_{\frac{2r}{2r-1}} = o(1),
\end{gather}
since $|\gamma(v_n)|^\theta - |\gamma(w_n)|^\theta \to 0$ strongly in $L_{\mathrm{loc}}^{\frac{2r}{2r-1}}(\mathbb{R}^N)$.
Here and in the following $C$ denotes some positive
constant independent of $n$, not necessarily the same one.
Similarly,%
\begin{multline*}
\Big\vert \int_{\mathbb{R}^{N}} \left( W\ast\left\vert
\gamma(v_{n}) \right\vert^{\theta} \right) \left\vert \gamma( v_{n})
\right\vert ^{\theta-2}\gamma(v_{n})
\gamma({w}_{n}) \\
{}-\int_{\mathbb{R}^{N}}\left( W\ast\left\vert \gamma(w_{n})\Big\vert
^{\theta}\right) \left\vert \gamma(w_{n})\right\vert^{\theta-2}
\gamma(w_{n}) \gamma({w}_{n}) \right\vert \\
\leq 2C |W|_r \left|\gamma(v_n)\right|^{\theta}_{\frac{2r \theta}{2r-1}} \left| |\gamma(v_n)|^\theta - |\gamma(w_n)|^\theta
\right|_{\frac{2r}{2r-1}} = o(1).
\end{multline*}
%
%
Therefore,
\begin{align*} \left\vert
I^{\prime}(v_{n})w_{n}-I^{\prime}(w_{n})w_{n} \right\vert & \leq C
\iint_{S_\epsilon} \left(|\nabla v_n|^2 + m^2 v_n^2 \right) dx \, dy +
\int_{ A_\epsilon} V(y)\gamma(v_n)^2\, dy +o(1).
\end{align*}
Set $u_n = (1- \xi) v_n$. Analogously we have
\begin{align*} \left\vert
I^{\prime}(v_{n})u_{n}-I^{\prime}(u_{n})u_{n} \right\vert & \leq C
\iint_{S_\epsilon} \left(|\nabla u_n|^2 + m^2 u_n^2 \right) dx \, dy +
\int_{ A_\epsilon} V(y)\gamma(u_n)^2\, dy +o(1).
\end{align*}
Therefore
\begin{equation}\label{f1} I^{\prime}(u_{n})u_{n} = O(\epsilon) + o(1)
\end{equation} and
\begin{equation}\label{f2} I^{\prime}(w_{n})w_{n}=O(\epsilon) + o(1).
\end{equation}
From (\ref{f1}), we derive that $I(u_n)= \frac{\theta -1 }{2 \theta}
\mathbb{D}(u_{n}) + O(\epsilon) + o(1) \geq O(\epsilon) + o(1)$.

Consider $t_n >0$ such that $I^{\prime}(t_n w_{n}) (t_n
w_{n})= 0$ for any $n$, namely
$$
t_n^{2 (\theta -1)} = \frac{\| w_n \|^2 +  \int_{\mathbb{R}^{N}}
V(y) \gamma(w_n)^2 \, dy}{\mathbb{D}(w_{n})}.
$$
From $(\ref{f2})$, we have that $t_n = 1 + O(\epsilon) + o(1)$.
Therefore from the characterization of $E_\alpha$ we have
\begin{align*}
c + o(1) &= I (v_n)= I(u_n) + I (w_n) + O(\epsilon)
\geq I (w_n) + O(\epsilon) + o(1) \\
& \geq I (t_n w_n) + O(\epsilon)
+ o(1) \geq E_\alpha + O(\epsilon) + o(1).
\end{align*}
As $n \to + \infty$, $ \epsilon \to 0$, we derive that $c \geq
E_\alpha$ which is a contradiction.  Hence, $c=0$ and
$v_{n}\rightarrow0$ strongly in $\h$.
\end{proof}
\begin{lemma} \label{lemnils}
Let $\{v_{n}\}_n$ be a sequence in $\h$ such
that $v_{n}\rightharpoonup v$ weakly in $\h$. The following hold:
  \begin{itemize}
    \item[(i)] $\mathbb{D}^{\prime}(v_{n})u\rightarrow\mathbb{D}^{\prime}(v) u$
for all $u\in H^{1}(\mathbb{R}_{+}^{N+1})$.
    \item[(ii)] After passing to a subsequence, there exists a
sequence $\{\widetilde{v}_{n}\}_n$ in $H^{1}(\mathbb{R}_{+}^{N+1})$ such
that $\widetilde{v} _{n} \rightarrow v$ strongly in $\h$,
\begin{align*}
\mathbb{D}(v_{n})-\mathbb{D}(v_{n}-\widetilde{v}_{n}) &
\rightarrow \mathbb{D}(v)\quad\text{in $\mathbb{R}$},\\
\mathbb{D}^{\prime}(v_{n})-\mathbb{D}^{\prime}(v_{n}-\widetilde{v}_{n})
& \rightarrow\mathbb{D}^{\prime}(v)\quad\text{in
$H^{-1}(\mathbb{R}_{+}^{N+1})$}.
\end{align*}
  \end{itemize}
\end{lemma}
\begin{proof} The proof is completely analogous to that of Lemma 3.5
in \cite{a}. The function $\widetilde{v}_{n}$ is the product of
$v_{n}$ with a smooth cut-off function, so $\widetilde{v}_{n}$ belongs
to $\h$ \ if $v_{n}$ does. We omit the details.
\end{proof}

\bigskip

\begin{proposition}
  \label{propPS}
The functional $I \colon \h \rightarrow\mathbb{R}$
satisfies the Palais-Smale condition $(PS)_{c}$ at each \(c< E_{V_{\infty}}\),
where $V_{\infty}:=\liminf_{\left\vert x\right\vert
\rightarrow\infty}V(x).$
\end{proposition}
\begin{proof} Let $v_{n}\in \h$ satisfy
\[ I(v_{n})\rightarrow c< E_{V_{\infty}}\quad\text{and}\quad
I'(v_{n})\rightarrow 0
\] strongly in the dual space $H^{-1}(\mathbb{R}_{+}^{N+1})$. Since
$\{v_{n}\}_n$ is bounded in $\h$ it contains a subsequence such that
$v_{n}\rightharpoonup v$ weakly in $\h$ and $\gamma(v_n)
\rightharpoonup \gamma(v)$ in $L^p(\mathbb{R} ^{N})$ for any $p \in[2,
2N/(N-1)]$.

By Lemma \ref{lemnils}, $v$ solves $(\ref{eq:1})$ and, after passing
to a subsequence, there exists a sequence $\{\widetilde{v}_{n}\}_n$ in
$\h$ such that $u_{n}:=v_{n}-\widetilde {v}_{n}\rightharpoonup0$
weakly in $\h$,%
\begin{align*}
I(v_{n})- I(u_{n}) & \rightarrow I(v)\quad\text{in
$\mathbb{R}$},\\ I'(v_{n}) - I' (u_{n}) & \rightarrow 0
\quad\text{strongly in $H^{-1}(\mathbb{R}_{+}^{N+1})$}.
\end{align*}
Hence, $I(v) = \frac{\theta- 2}{2 \theta} \mathbb{D}(v)
\geq 0$,
\[
I(u_{n})\rightarrow c-I (v)\leq c,\text{ \ \ and \ \ \ }
I'(u_{n})\rightarrow0
\]
strongly in $H^{-1}(\mathbb{R}_{+}^{N+1})$. By Lemma
\ref{lemPSto0}\ a subsequence of $\{u_{n}\}_n$ converges strongly to $0$
in $\h$. This implies that a subsequence of $\{ v_{n}\}_n $
converges strongly to $v$ in $\h$.
\end{proof}

\section{Mountain Pass Geometry}

Let us consider the limit problem
\begin{equation}\label{6}
  \begin{cases} -\Delta v + m^2 v =0 &\text{in
$\mathbb{R}^{N+1}_{+}$}\\ -\frac{\partial v}{\partial x} = - V_\infty
v + ( W*|v|^{\theta})|v|^{\theta-2}v &\text{in $\mathbb{R}^N
= \partial \mathbb{R}^{N+1}_{+}$}
  \end{cases}
\end{equation}
where $V_{\infty}:=\liminf_{\left\vert x\right\vert
\rightarrow\infty}V(x)>0$.
By Theorem \ref{th:CZN}, the first mountain pass value $E_{V_\infty}$ of
the functional $J_{V_\infty}$ associated to problem (\ref{6}) is
attained at a positive function $\omega_{\infty} \in
H^{1}(\mathbb{R}^{N+1}_+)$, which is symmetric $\omega_\infty (x, gy)=
\omega_\infty(x,y)$ for all $g \in O(N)$, $x>0$, $y \in
\mathbb{R}^{N}$.  Moreover, since $V_\infty>0$, we are allowed to choose $\sigma=0$, and
there exists $C>0$ such that
\begin{equation}\label{eq:13}
0 < \omega_\infty (x,y) \leq C e^{-m \sqrt{x^2 + |y|^2}}
\end{equation}
for all $(x,y) \in [0, + \infty) \times \mathbb{R}^N$.
In particular, $\gamma(\omega_\infty)$ is radially symmetric in
$\mathbb{R}^{N}$ and
$$
0 < \gamma(\omega_\infty)(y) \leq C e^{-m |y| }
$$
for any $y \in \mathbb{R}^{N}$. As in Theorem \ref{th:CZN}, a bootstrap procedure shows that
$\omega_\infty \in C^\infty([0,+\infty) \times \mathbb{R}^N)$.
\begin{lemma} We have
\begin{equation}\label{decgrad}
\left\vert
\nabla\omega_{\infty}(z)\right\vert =O(e^{-{m}\left\vert z
\right\vert}) \text{ \ \ as }\left\vert z\right\vert
\rightarrow\infty.
\end{equation}
\end{lemma}
\begin{proof} We consider the equation
\begin{equation*} \sqrt{-\Delta +m^2} u + V_\infty u = \left( W*
|u|^\theta \right) |u|^{\theta-2}u \quad\text{in $\mathbb{R}^N$}
\end{equation*} satisfied by $\omega_\infty$.  For any index $i
=1,2,\ldots,N$ we write $v_i=\partial \omega_\infty / \partial y_i$
and observe that $v_i$ satisfies
\begin{equation} \label{eq:14} \sqrt{-\Delta + m^2} v_i + V_\infty v_i
= \\ \theta (W * \omega_\infty^{\theta -1} v_i)
\omega_\infty^{\theta-1} + (\theta -1) (W*\omega_\infty^\theta )
\omega_\infty^{\theta -2} v_i
\end{equation}
or, equivalently,
\begin{align*} &-\Delta v_i + m^2 v_i = 0 \quad\text{in
$\mathbb{R}_{+}^{N+1}$} \\ &-\frac{\partial v_i}{\partial x} =
-V_\infty v_i + \theta (W * \omega_\infty^{\theta -1} v_i)
\omega_\infty^{\theta-1} + (\theta -1) (W*\omega_\infty^\theta )
\omega_\infty^{\theta -2} v_i \quad\text{in $\mathbb{R}^N$}.
\end{align*}

The differentiation of the equation is allowed by the
regularity of the solution $\omega_\infty$, see \cite[Theorem
3.14]{CN}. Moreover, $\omega_\infty \in L^p(\mathbb{R}_{+}^{N+1})$ for
any $p>1$, because it is bounded and decays exponentially fast at
infinity. By elliptic regularity, $\omega_\infty \in
W^{2,p}(\mathbb{R}_{+}^{N+1})$ for any $p>1$, and in particular $v_i
\in L^p(\mathbb{R}_{+}^{N+1})$ for any $p>1$. An interpolation
estimate shows that $\omega_\infty^{\theta-1}v_i \in
L^p(\mathbb{R}_{+}^{N+1})$ for any $p>1$. Then the convolution
$W*(\omega_\infty^{\theta-1} v_i) \in L^\infty(\mathbb{R}_{+}^{N+1})$,
and the term
$(W*(\omega_\infty^{\theta-1}v_i))\omega_\infty^{\theta-1} \in
L^2(\mathbb{R}_{+}^{N+1})$ by the summability properties of
$\omega_\infty$. The term $(W*\omega_\infty^\theta )
\omega_\infty^{\theta -2} v_i \in L^2 (\mathbb{R}_{+}^{N+1})$
trivially.

Now the proof of \cite[Theorem 3.14]{CN} shows that $v_i(x,y)
\to 0$ as $x+|y| \to +\infty$. A
comparison with the function $e^{-m\sqrt{x^2+|y|^2}}$ as in
\cite[Theorem 5.1]{CN} shows the validity of
(\ref{decgrad}).\end{proof}
Fix $\varepsilon \in (0, \frac{2m - k}{2m +k})$. For $R>0$, we consider a symmetric
cut off function $\xi_R \in C^{\infty}(\mathbb{R}_+^{N+1})$, namely
$\xi_R(x, gy) = \xi_R (x,y)$ for all $g \in O(N)$, $x >0$, $y \in
\mathbb{R}^{N}$ such that $\xi_R (z)= 0$ if $|z| \geq R$ and $\xi_R
(z)= 1$ if $|z| \leq R (1 -\epsilon) $ and $\xi_R(z)\in\lbrack0,1]$
for all $z \in\mathbb{R}_{+}^{N+1}.$

Let us define $\omega^R(z) := \omega_\infty(z) \xi_R(z)$, for any $z
\in \mathbb{R}_+^{N+1}$.
\begin{lemma}\label{lemassymp}
As $R\rightarrow\infty$,
\begin{align}
\left\vert\iint_{\mathbb{R}^{N+1}_+}  \left\vert
\nabla\omega_{\infty}\right\vert ^{2}-\left\vert
\nabla\omega^{R}\right\vert ^{2}\right\vert & =O(R^{N-1}e^{-2 {m}(1-\varepsilon)R}),\label{b}\\ \left\vert
\mathbb{D}(\omega_{\infty})-\mathbb{D}(\omega ^{R})\right\vert &
=O(R^{N-1}e^{- \theta m (1-\varepsilon)R}). \label{c}%
\end{align}
\end{lemma}
\begin{proof} The proof of (\ref{b}) is standard. Indeed, using (\ref{decgrad}) and cylindrical coordinates in $\mathbb{R}_{+}^{N+1}$,
\begin{align*}
\left\vert \iint_{\mathbb{R}_{+}^{N+1}} |\nabla \omega^R|^2 - |\nabla \omega_\infty|^2 \right\vert &\leq C \iint_{\{z \in \mathbb{R}_{+}^{N+1} \mid (1-\varepsilon)R<|z| \}} |\nabla \omega_\infty|^2  \\
&\leq C_1 \iint_{\{z \in \mathbb{R}_{+}^{N+1} \mid (1-\varepsilon)R<|z| \}} e^{-2m |z|}\, dz \\
&\leq C_1 R^{N-1}e^{-2m(1-\varepsilon)R}.
\end{align*}
To prove (\ref{c}), we recall that $W = W_1 +W_2 \in
L^r(\mathbb{R}^N) + L^\infty (\mathbb{R}^N)$. The difference
$\mathbb{D}(\gamma (\omega_\infty))-\mathbb{D}(\gamma(\omega^R))$ can
be split in two parts, the one with $W_1$ and the one with $W_2$. The
former can be estimated as follows:
\begin{multline*}
\left|
\mathbb{D}(\gamma(\omega_\infty))-\mathbb{D}(\gamma(\omega^R)) \right| \\
\leq \int_{\mathbb{R}^N \times \mathbb{R}^N} \left|
\gamma(\omega_\infty)(x)|^\theta |\gamma (\omega_\infty)(y)|^\theta -
|\gamma(\omega^R)(x)|^\theta |\gamma (\omega^R)(y)|^\theta \right|
W_1(x-y)\, dx \, dy \\
\leq 2 \int_{\mathbb{R}^N \times
\mathbb{R}^N} W_1(x-y) |\gamma(\omega_\infty)(x)|^\theta \left|
|\gamma(\omega_\infty)(y)|^\theta - |\gamma (\omega^R)(y)|^\theta
\right| dx \, dy \\
\leq 2 \left\Vert |\gamma(\omega_\infty)|^\theta
- |\gamma(\omega^R)|^\theta \right\Vert_{\frac{2r}{2r-1}} \left\Vert
\gamma(\omega_\infty) \right\Vert_{\frac{2r\theta}{2r-1}}^\theta
\left\Vert W_1 \right\Vert_r \\
\leq C \left(
\int_{(1-\varepsilon)R}^\infty t^{N-1}e^{-m \frac{2r\theta}{2r-1} t}
\, dt \right)^{\frac{2r-1}{2r}} = C_2 R^{N-1}e^{- \theta m (1-\varepsilon)R}.
\end{multline*}
The latter is simpler, since we use directly the
$L^\infty$-norm of $W_2$.
\end{proof}
For $s \in\mathbb{R}^{N}$, set $R_s := \frac{k + 2m}{4m} \, |s|$.  Since $k \in \bigl(0, 2m \bigr)$, it results that $R_s \in (0, |s|)$. Hence $|s| - R_s
\to + \infty$, as $|s| \to + \infty$.
With this notation, we define the function
\[ \omega^{R_s}_s(z) := \omega_\infty(x, y -s) \xi_{R_s}(x, y -s)
\] where $z=(x,y) \in \mathbb{R}^{N+1}$.

\begin{lemma}\label{lemub}
There exist $\varrho_{0}$, $d_{0}\in(0,\infty)$ such that
\[
I (t(\omega^{R_s}_{s}))\leq E_{V_{\infty}}-d_{0}%
e^{-k\left\vert y\right\vert }\quad\text{for all $t\geq 0$},
\]
provided that $\left\vert s\right\vert \geq\varrho_{0}$
\end{lemma}

\begin{proof}
For $u \in \h$ we have by (\ref{eq:21}) that
$\max_{t\geq0}I(tu)=I(t_{u}u)$ if and only if
\[
t_{u}=\left( \frac{\| u\|^2 + \int_{\mathbb{R}^{N}} V(y)\gamma(u)^ 2 \,
dy}{\mathbb{D}(u)}\right)^{1/(2\theta-2)}.
\]
Indeed
\begin{equation}\label{eq:22}
\| u\|^{2} +
\int_{\mathbb{R}^{N}} V(y) \gamma(u)^2 dy
\geq
\left( 1 - \frac{V_0}{m} \right) \int_{\mathbb{R}_{+}^{N+1}} |\nabla u|^2 + {m(m-V_0)} \int_{\mathbb{R}_{+}^{N+1}} |u|^2>0.
\end{equation}
So, since $\omega_{\infty}^{R_{s}}\rightarrow \omega_{\infty}$ in
$\h$ as $\left\vert s\right\vert \rightarrow \infty,$ and taking into
account that $I_{V_\infty}(\omega_\infty)=
\max_{t\geq0}I_{V_\infty}(t(\omega_{\infty}))$ there exist
$0<t_{1}<t_{2}<+ \infty$ such that
\[
\max_{t\geq0}I(t(\omega^{R_{s}}_{s}))=\max_{t_{1}\leq t\leq
t_{2}}I(t(\omega^{R_{s}}_{s}))
\]
for all large enough $\left\vert s\right\vert $.

Let $t\in\lbrack t_{1},t_{2} \rbrack$. Write $V=V^{+}-V^{-}$, where $V^{+}(x)=\max \{V(x),0\}$
and $V^{-}(x)=\max\{-V(x),0\}$, and remark that the assumption $V_\infty>0$ implies $V(x)=V^{+}(x)$
whenever $|x|$ is sufficiently large. Assumption $(V_{2})$ yields therefore
\begin{align*}
\int_{\mathbb{R}^{N}}V( y) (t \gamma(\omega^{R_{s}
}_{s}))^{2} (y) \, dy & \leq t^{2}\int_{\left\vert y \right\vert \leq
R_{s}} V^{+}(y +s)(\gamma(\omega^{R_{s}}))^{2}(y)dy\\
& \leq
t^{2}\int_{\left\vert y\right\vert \leq R_{s}}(V_{\infty}
-c_{0}e^{-k \left\vert y+s \right\vert
})(\gamma(\omega_{\infty}))^{2}(y)dx\\
&\leq\int_{\mathbb{R}^{N}}V_{\infty}\left( t \gamma(\omega_{\infty})
\right) ^{2}-\left( c_{0}t_{1}^{2}\int_{\left\vert y \right\vert
\leq1}e^{-k \left\vert y\right\vert }
(\gamma(\omega_{\infty}))^{2}(y)dy\right) e^{-k \left\vert
s\right\vert }
\end{align*} for $\left\vert s\right\vert $ large enough.

Therefore, using Lemma \ref{lemassymp}, we get
\begin{align*}
I(t (\omega^{R_{s}})_{s}) & =\frac{1}{2}\left\Vert
t(\omega^{R_{s}})_{s}\right\Vert^{2} + \frac{1}{2}
\int_{\mathbb{R}^{N}} V(y)\left( t \gamma(\omega_{\infty}) \right)^{2}
dy -\frac{1}{2 \theta} \mathbb{D}(t \omega^{R_{s}}_{s})\\
& \leq
\frac{1}{2}\left\Vert t \omega_\infty \right\Vert^{2} + \frac{1}{2}
\int_{\mathbb{R}^{N}} V_{\infty}\left( t \gamma(\omega_{\infty})
\right)^{2} dy -\frac{1}{2 \theta}
\mathbb{D}(t\omega_{\infty})\\
&\qquad -Ce^{-k \left\vert s\right\vert
}+O(R_s^{N-1} e^{-2 m (1-\varepsilon)R_{s}})\\
&
\leq\max_{t\geq0}I_{V_{\infty}}(t\omega_{\infty})-d_{0}e^{-\kappa\left\vert
s\right\vert }\\
& =E_{V_{\infty}}-d_{0}e^{-k \left\vert
s\right\vert }%
\end{align*}
for sufficiently large $\left\vert s\right\vert $,
because our choices of $\varepsilon$ and $R_{s}$ guarantee that $2 m (1-\varepsilon )R_{s}>k \left\vert s\right\vert$.
\end{proof}

\section{Proof of theorem \ref{th:main}}

The proof of Theorem \ref{th:main} is now immediate. The Euler
functional $I$ satisfies the geometric assumptions of the Mountain
Pass Theorem (see \cite{ar}) on $H^1(\mathbb{R}_{+}^{N+1})$. Since it also satisfies the Palais-Smale
condition, as we showed in the previous sections, we conclude that $I$
possesses at least a critical point $v \in H^1(\mathbb{R}_{+}^{N+1})$. In addition,
\[
I(v) = c = \inf_{\gamma \in \Gamma} \max_{0 \leq t \leq 1} I(\gamma (t)),
\]
where $\Gamma = \{\gamma \in C([0,1],H^1(\mathbb{R}_{+}^{N+1})) \mid \gamma(0)<0,\,
I(\gamma(1))<0\}$.

To prove that $v \geq 0$, we notice that, reasoning as in (\ref{eq:22}), the map $t \mapsto I(tw)$ has
one and only one strict maximum point at $t=1$ whenever $w\in
H^1(\mathbb{R}_{+}^{N+1})$ is a critical point of $I$. Since $I(|w|)
\leq I(w)$ for all $w \in H^1(\mathbb{R}_{+}^{N+1})$, and
\[
I(t |w|) \leq I(t w) < I(w) \quad\text{for every $t >0$, $t \neq 1$},
\]
we conclude that
\[
c \leq \sup_{t \geq 0} I(t|v|) \leq I(v) = c.
\]
We claim that $|v|$ is also a critical point of $I$. Indeed,
otherwise, we could deform the path $t \mapsto t|v|$
into a path $\gamma \in \Gamma$ such that $I(\gamma(t))<c$ for
every $t \geq 0$, a contradiction with the definition of $c$.

\section{Further properties of the solution}

We collect in the next statement some additional features of the weak
solution found above.

\begin{theorem} \label{th:7.1}
Let $u$ be the solution to equation (\ref{eq:1})
provided by Theorem \ref{th:main}. Then $u \in C^\infty(\mathbb{R}^N)
\cap L^q (\mathbb{R}^N)$ for every $q \geq 2$. Moreover,
\begin{equation} \label{eq:30}
0<u(y) \leq C e^{-m |y|}.
\end{equation}
\end{theorem}
\begin{proof}
The regularity of $u$ can be established by mimicking
  the proofs in Section 3 of \cite{CN}.
  The potential function $V$ is
  harmless, being bounded from above and below.

To prove the exponential decay at infinity, we introduce a comparison function
\[
 W_R(x,y) = C_R e^{-m\sqrt{x^2+|y|^2}}, \quad \text{for every $(x,y) \in \mathbb{R}_{+}^{N+1}$},
\]
and we will fix $R>0$ and $C_R>0$ in a suitable manner. We also introduce the notation
\begin{align*}
 B_R^{+} &= \left\{ (x,y) \in \mathbb{R}_{+}^{N+1} \mid \sqrt{x^2+|y|^2} < R \right\} \\
 \Omega_R^{+} &= \left\{ (x,y) \in \mathbb{R}_{+}^{N+1} \mid \sqrt{x^2+|y|^2} > R \right\} \\
 \Gamma_R &= \left\{ (0,y) \in \partial \mathbb{R}_{+}^{N+1} \mid |y| \geq R \right\}.
\end{align*}
It is easily seen that
\begin{equation*}
 \left\{
 \begin{array}{ll}
  -\Delta W_R + m^2 W_R \geq 0 &\text{in $\Omega_R^{+}$}\\
  -\frac{\partial W_R}{\partial x} =0 &\text{on $\Gamma_R^{+}$}.
 \end{array}
 \right.
\end{equation*}
Call $w(x,y)=W_R(x,y)-v(x,y)$, and remark that $-\Delta w + m^2 w \geq 0$ in $\Omega_R^{+}$. If $C_R = e^{mR} \max_{\partial B_R^{+}} v$, then
$w \geq 0$ on $\partial B_R^{+}$ and $\lim_{x+|y| \to +\infty} w(x,y)=0$. We claim that $w \geq 0$ in the closure
$\overline{\Omega_R^{+}}$.

If not, $\inf_{\overline{\Omega_R^{+}}} w <0$, and the strong maximum principle provides a point $(0,y_0) \in \Gamma_R$ such that
\[
 w(0,y_0)=\inf_{\overline{\Omega_R^{+}}} w < w(x,y)\quad\text{for every $(x,y) \in  \Omega_R^{+}$}.
\]
For some $0<\lambda<m$, we introduce $z(x,y)=w(x,y)e^{\lambda x}$. As before, $\lim_{x + |y| \to +\infty} z(x,y)=0$ and $z \geq 0$
on $\partial B_R^{+}$. Since
\[
 0 \leq -\Delta w +m^2 w = e^{-\lambda x} \left( -\Delta z + 2 \lambda \frac{\partial z}{\partial x} + (m^2-\lambda^2)z,  \right)
\]
the strong maximum principle applies and yields that $\inf_{\Gamma_R} z = \inf_{\overline{\Omega_R^{+}}} z < z(x,y)$ for every
$(x,y) \in \Omega_R^{+}$. Therefore $z(0,y_0) = \inf_{\Gamma_R} z = \inf_{\Gamma_R} w <0$. Hopf's lemma now gives
\[
 -\frac{\partial w}{\partial x} (0,y_0)	-\lambda w(0,y_0)<0.
\]
But this is impossible. Indeed
\[
 -\frac{\partial w}{\partial x} (0,y_0) = -V(y_0) v(0,y_0) - (W*|v|^\theta)|v(0,y_0)|^{\theta -2}v(0,y_0),
\]
and hence
\[
 -\frac{\partial w}{\partial x} (0,y_0) -\lambda v(0,y_0)= -\lambda v(0,y_0)-V(y_0) v(0,y_0) - (W*|v|^\theta)|v(0,y_0)|^{\theta -2}v(0,y_0).
\]
Recall that $v(0,y_0)<0$ and $\lambda >0$; if we can show that $$-V(y_0) v(0,y_0) - (W*|v|^\theta)|v(0,y_0)|^{\theta -2}v(0,y_0) \geq 0,$$ we
will be done. First of all, we recall that (see \cite[pag. 70]{CN} and also \cite[Lemma 2.3]{css})
\[
 \lim_{|y| \to +\infty} (W*|v|^\theta)|v(0,y)|^{\theta -2}v(0,y)=0,
\]
since $\lim_{|y| \to +\infty} W(y)=0$. So we pick $R>0$ so large that  $|(W*|v|^\theta)|v(0,y_0)|^{\theta -2}v(0,y_0)|$ is very small. Choosing
$R$ even larger, we can also assume that $V(y_0)>0$, since $V_\infty>0$. Hence $-V(y_0) v(0,y_0) - (W*|v|^\theta)|v(0,y_0)|^{\theta -2}v(0,y_0) \geq 0$, and the proof
is finished.

To summarize, we have proved that, whenever $x+|y|$ is sufficiently large, then
\[
 v(x,y) \leq W_R(x,y),
\]
and hence the validity of (\ref{eq:30}).

\vspace{1cm}

\noindent{\bfseries Acknowledgements.} The first author is supported by  MIUR Project (PRIN 2009) \emph{"Variational and topological methods in the study of nonlinear phenomena"} and by GNAMPA (INDAM) Project 2013
\emph{"Problemi differenziali di tipo ellittico nei
fenomeni fisici non lineari"}. The second author is supported by MIUR Project (PRIN 2009) \emph{"Teoria dei punti critici e metodi perturbativi per equazioni differenziali nonlineari"}.

\end{proof}


\end{document}